\documentclass[12pt,a4paper, leqno]{amsart}

\usepackage{amsmath, amsfonts, amssymb, amsthm, array, stmaryrd, graphicx, hyperref, mathrsfs, eucal, caption, soul}

\usepackage{amsrefs}
\AtBeginDocument{%
	\def\MR#1{}
}

\usepackage{latexsym}
\usepackage{verbatim}
\usepackage{amscd}

\usepackage[a4paper, total={150mm, 237mm}, left=30mm, top=30mm]{geometry}

\numberwithin{equation}{section}

\newtheorem{prop}{Proposition}[section]
\newtheorem{theorem}[prop]{Theorem}

\newtheorem{remark}[prop]{Remark}

\def\begeq{\begin{equation}}
\def\endeq{\end{equation}}

\def\<{\langle}
\def\>{\rangle}
\def\({\left(}
\def\){\right)}

\def\Ric{{\rm Ric}}

\def\p{\partial}

\begin{document}

\title[The Modified K\'ahler-Ricci flow, II]{The Modified K\"ahler-Ricci Flow, II}

\author{Haotian Wu}
\address{School of Mathematics and Statistics, The University of Sydney, NSW 2006, Australia}
\email{haotian.wu@sydney.edu.au}

\author{Zhou Zhang}
\address{School of Mathematics and Statistics, The University of Sydney, NSW 2006, Australia}
\email{zhou.zhang@sydney.edu.au}


\keywords{Modified K\"ahler-Ricci flow, degenerate Calabi-Yau equation.}

\subjclass[2010]{53C44, 32W20}

\maketitle

\begin{abstract} We improve the understanding of both finite time and infinite time singularities of the modified K\"ahler-Ricci flow as initiated by the second author of this paper in \cite{modified}. This is done by relating the modified K\"ahler-Ricci flow with the recent studies on the classic K\"ahler-Ricci flow and the degenerate complex Monge-Amp\`ere equation.
\end{abstract}

\section{Introduction}\label{intro}

We begin by recalling the set up of a flow of K\"ahler-Ricci type which is first introduced in Section 9.4.1 of \cite{thesis} by the second author of this paper. Let $\omega_0$ be a K\"ahler metric over a closed manifold $X$ with $\dim_{\mathbb{C}}X=n\geqslant 2$, and $\omega_\infty$ be a smooth real closed $(1,1)$-form over $X$.

Set $\omega_t=\omega_\infty+e^{-t}(\omega_0-\omega_\infty)$ and consider the following parabolic flow at the level
of metric potential in spacetime:
\begin{equation}
\label{eq:M-SKRF}
\frac{\partial u}{\partial t}={\rm log}\frac{(\omega_t+\sqrt{-1}\partial\bar{\partial}u)^n}{\Omega}, \quad\quad u(0,\cdot)=0,
\end{equation}
where $\Omega$ is a smooth volume form over $X$.

Let $\widetilde{\omega}_t=\omega_t+\sqrt{-1}\partial\bar
{\partial}u$. Then the corresponding flow at the level of metric (or equivalently, form) is obtained by taking $\sqrt{-1}\p\bar\p$ of both sides of equation \eqref{eq:M-SKRF}:
\begin{equation}
\label{eq:M-KRF} 
\frac{\partial\widetilde{\omega}_t}{\partial t}=-{\rm Ric}(\widetilde{\omega}_t)+{\rm Ric}(\Omega)-e^{-t}
(\omega_0-\omega_\infty),\quad\quad \widetilde{\omega}_0=\omega_0,
\end{equation}
where the form ${\rm Ric}(\Omega)$ is a natural generalisation of the Ricci form for a K\"ahler metric, using the volume form $\Omega$ instead of the volume form for a K\"ahler metric. This makes use of the classic calculation of Ricci form in K\"ahler geometry.

The flow \eqref{eq:M-SKRF} arises naturally as a variation of the more classic K\"ahler-Ricci flow considered earlier, e.g. in \cite{t-znote}, especially at the level of scalar potential flow. In \cite{modified}, the second author of this paper obtained the optimal existence result for this flow, i.e. the classical (i.e. smooth metric) solution exists as long as the class $[\omega_t]$ remains K\"ahler, and proved strong convergence as $t\to\infty$ when the limiting class at time infinity, denoted by $[\omega_\infty]$, is K\"ahler. Together with the subsequent work \cite{yuan-convergence} by Yuan, we have local smooth convergence when $[\omega_\infty]$ is big and the corresponding limiting class (at either finite or infinite time) is semi-ample.

One main motivation for introducing the modified K\"ahler-Ricci flow \eqref{eq:M-KRF}, or equivalently \eqref{eq:M-SKRF}, is to study the corresponding degenerate complex Monge-Amp\`ere equation, i.e. the degenerate Calabi-Yau equation
\begin{equation}
\label{eq:D-CY}
(\omega_\infty+\sqrt{-1}\p\bar\p U)^n=\Omega,
\end{equation}
when $[\omega_\infty]$ is no longer K\"ahler, i.e. degenerate. The flow \eqref{eq:M-KRF} can be viewed as a sophisticated way of setting
up the method of continuity. Of course, one can use alternative ways to set up the continuity method, see for example the discussion in Section 8.5.3 of \cite{thesis}. Many independent works, e.g., \cite{tosatti} by Tosatti and \cite{yuguang} by Zhang, were carried out in this direction. For the collapsed case, there is the recent works by Hein-Tosatti \cite{hein-tosatti} in the elliptic setting focusing on higher order regularity. 
 
In this paper, we focus on the flow approach, i.e. (\ref{eq:M-KRF}) and equivalently (\ref{eq:M-SKRF}). Many techniques effective for the classic K\"ahler-Ricci flow can be adapted to the current setting. 

In principle, we try to avoid extra assumptions on $[\omega_\infty]$ or $K_X=-c_1(X)$, resulting in tremendous advantage when comparing and combining with various earlier works on K\"ahler-Ricci flows and even Ricci flow in general. Under the modified K\"ahler-Ricci flow (\ref{eq:M-KRF}), the K\"ahler class  
$$[\widetilde\omega_t]=[\omega_t]=e^{-t}[\omega_0]+(1-e^{-t})[\omega_\infty]$$
starts from the K\"ahler class $[\omega_0]$ at $t=0$ and would in general become non-K\"ahler at time
\begin{align}\label{defT}
T=\sup\,\{t \,|\, e^{-t}[\omega_0]+(1-e^{-t})[\omega_\infty] ~\text{is K\"ahler}\}\in (0, \infty]
\end{align} 
with $[\omega_T]$ on the boundary of the K\"ahler cone of $X$. With a slight abuse of notation from algebraic geometry, we say $[\omega_T]$ is nef. (i.e. numerically effective). In the same spirit as \cite{t-znote}, we have the following optimal existence result. 

\begin{theorem}
The modified K\"ahler-Ricci flow (\ref{eq:M-KRF}), or equivalently (\ref{eq:M-SKRF}), exists exactly for $t\in [0, T)$, where $T$ is defined in \eqref{defT}.  
\end{theorem}

This result holds for very general flow of K\"ahler-Ricci type. The proof is described in \cite{modified}. 

Now we consider the two cases of $T<\infty$ and $T=\infty$, respectively.

When $T<\infty$, the classical flow, i.e. the time slice being smooth and metric, has to stop because it is no longer K\"ahler at $T$ even for cohomology, and so there must be a finite time singularity. In this case, we have the following geometric characterisation of the singularity and also justify the flow weak limit for cases of interest.

\begin{theorem}
\label{th:finite-time-singularity}
Consider the modified K\"ahler-Ricci flow (\ref{eq:M-KRF}) with $T<\infty$. We have 
$$\lim_{T_0\to T^-}\inf_{[0, T_0]\times X}\log\frac{\widetilde\omega^n_t}{\Omega}\to-\infty.$$  
Furthermore, if there is $\varphi\in PSH_{\omega_T}(X)\cap L^\infty(X)$, then we have the flow weak limit
$$\widetilde\omega_t\to \omega_T+\sqrt{-1}\p\bar\p U$$
as $t\to T^-$ for $U\in PSH_{\omega_T}(X)\cap L^\infty(X)$. 
\end{theorem}

When $T=\infty$, there are two possibilities. Firstly, $[\omega_\infty]$ is K\"ahler. It is shown in \cite{modified} that the flow converges strongly to the smooth solution of 
$$\Ric(\omega_\infty+\sqrt{-1}\p\bar\p U)=\Ric(\Omega), \quad \text{i.e.} \quad (\omega_\infty+\sqrt{-1}\p\bar\p U)^n=A\Omega$$
for $A=\frac{[\omega_\infty]^n}{\int_X\Omega}$, whose existence and uniqueness is nothing but the classic Calabi-Yau Theorem \cite{calabi-yau}. Secondly, $[\omega_\infty]$ is on the boundary the K\"ahler cone. Then it is impossible to have strong convergence to a metric because of cohomology obstruction. There must be singularity forming at $T=\infty$, i.e. the infinite time singularity case. The limit is then a generalised solution to the degenerate Calabi-Yau equation (\ref{eq:D-CY}). 

Meanwhile, we clearly have the global volume $[\omega_T]^n\geqslant 0$ as the limit of $[\omega_t]^n$. The case of $[\omega_T]^n>0$ is called non-collapsed, which more or less corresponds to the case of $[\omega_T]$ being big from algebraic geometry. Then we have local smooth convergence of the flow away from the stable base locus set of $[\omega_T]$ by the results in \cite{yuan-convergence, modified}. From pluripotential theory in \cite{demailly-pali, ey-gu-ze, zzo}, we also know the uniqueness, boundedness and even continuity of the limit metric potential. In summary, we have the following theorem. 

\begin{theorem}
\label{weak-convergence}

Under the modified K\"ahler-Ricci flow (\ref{eq:M-KRF}) with $[\omega_T]$ semi-ample and big, $\widetilde\omega_t$ converges smoothly out of the stable base locus set of $[\omega_T]$ with the limit 
$$\widetilde\omega_T=\omega_T+\sqrt{-1}\p\bar\p U, \qquad U\in PSH_{\omega_T}(X)\cap C^0(X).$$ 
Moreover, if $T=\infty$, then $\frac{\p u}{\p t}\geqslant -C$ uniformly for all time, i.e. the volume form has a uniform
positive lower bound, and the limit is the unique weak solution to the degenerate Calabi-Yau equation (\ref{eq:D-CY}). 

\end{theorem}

However, the above theorem is insufficient to control the global metric geometry of the flow. By adapting the notions and arguments in \cite{song-weinkove}, we have the following control of the global geometry along the flow. 

\begin{theorem} 
\label{th:metric-control}

Consider a K\"ahler manifold $X$ with $\Ric(\Omega)\leqslant 0$ (i.e. $c_1(X)$ non-positive). For the modified K\"ahler-Ricci flow (\ref{eq:M-KRF}) with $[\omega_\infty]$ generating a canonical surgical contraction, if $\omega_0$ is sufficiently positive, then the flow has uniform bounded diameter and converges in the Gromov-Hausdorff sense to the metric compactification of the smooth flow limit out of the stable base locus set of $[\omega_\infty]$.     

\end{theorem}

\begin{remark}

The notion of a canonical surgical contraction is from \cite{song-weinkove} and will be recalled in the proof of Theorem \ref{th:metric-control}. For the limiting equation (\ref{eq:D-CY}), it is assumed that $\Ric(\Omega)=0$. Together with the uniqueness of the solution for the limiting equation (\ref{eq:D-CY}), the assumptions in Theorem \ref{th:metric-control} is not restrictive. 

\end{remark}

In comparison, the case of $[\omega_T]^n=0$ is called collapsed, which often corresponds to the case of $[\omega_T]=F^*[\omega]$ for a holomorphic map $F$ with the image of dimension $n-r<\dim_\mathbb{C}X=n$, i.e. a fibration map of generic fibre dimension $r\in \{1, \cdots, n\}$. In this case, we have the following result when $T=\infty$. 

\begin{theorem}
\label{th:collapsed-volume}
For the modified K\"ahler-Ricci flow (\ref{eq:M-SKRF}) with a collapsed infinite time singularity, by properly choosing $\Omega$, we have positive constants $C$ and $A$ such that for any $S>0$,
\begin{align*}-\frac{r^2}{2}t^2 & \leqslant u\leqslant C+At-\frac{r^2}{2}t^2, \\
-\(\frac{A}{1-e^{-S}}+r\)t-C(S) & \leqslant \frac{\p u}{\p t}\leqslant \(\frac{A}{e^S-1}-r\)t+C(S),
\end{align*}
where the constant $C(S)$ depending on $S$ over $X\times [0,\infty)$. 
\end{theorem}

Theorem \ref{th:collapsed-volume} describes in the collapsed case, the way of $u$ going to $-\infty$ as $t\to\infty$, and provides, by choosing $S$ as large as one wants, an almost optimal rate of volume collapsing. 

This paper is organised as follows. In Section \ref{prelim}, we fix the notation and derive preliminary estimates for the modified K\"ahler-Ricci flow \eqref{eq:M-SKRF}. Section \ref{finite-time-sing} is devoted to the proof of Theorem \ref{th:finite-time-singularity}. We establish Theorem \ref{th:metric-control} in Section \ref{geom-lim}. Theorem \ref{th:collapsed-volume} is proved in Section \ref{deg-MA} and the Appendix.

\noindent {\bf Acknowledgements.} H. Wu is partially supported by the ARC grant DE180101348; Z. Zhang is partially supported by the ARC grant FT150100341.


\section{Preliminary estimates}\label{prelim}

\noindent {\bf Notation.} In this paper, we adopt the convention that $C$ denotes a positive constant that changes from line to line. When it matters, we indicate the dependency of $C$ on other quantities. We denote by $\Delta$ the Laplacian with respect to the flow metric. The expression $\<\Phi, \Psi\>$ stands for taking the trace of $\Psi$ with respect to a metric $\Phi$ (at least where it is considered), which is equivalent to taking the inner product with respect to $\Phi$.

Under the flow (\ref{eq:M-SKRF}), at a spatial maximum $u_{\max}(t)$ of $u$, we have 
$$\frac{d u_{\max}}{d t}\leqslant \log\(\max\frac{\omega^n_t}{\Omega}\)=C$$
uniformly for $[0, \infty)$. So by considering $u-Ct$ as the new $u$, or equivalently changing $\Omega$ to $e^C\Omega$, we have the decrease of $u_{\max}$ without changing the metric flow. Since the initial value of $u$ is $0$, we have
$$u\leqslant 0.$$ 

Taking the $t$-derivative of equation (\ref{eq:M-SKRF}), we get 
\begin{equation}
\label{eq:t-derivative}
\frac{\p}{\p t}\(\frac{\p u}{\p t}\)=\<\widetilde\omega_t, \frac{\p \widetilde\omega_t}{\p t}\>=\Delta\(\frac{\p u}{\p t}\)-e^{-t}\<\widetilde\omega_t, \omega_0-\omega_\infty\>.
\end{equation}
Another $t$-derivative gives 
\begin{align}
\label{eq:t-2nd-derivative}
\frac{\partial}{\partial t}\(\frac{\partial^2 u}{\partial t^2}\)  & =\<\widetilde{\omega}_t, \frac{\partial^2\widetilde{\omega}_t}{\partial t^2}\>-\(\frac{\partial\widetilde{\omega}_t}{\partial t}, \frac{\partial\widetilde{\omega}_t}{\partial t} \)_{\widetilde{\omega}_t} \\
& \leqslant\Delta \(\frac{\partial ^2 u}{\partial t^2}\)+e^{-t}\<\widetilde{\omega}_t, \omega_0-\omega_\infty\>. \notag
\end{align}

Applying the maximum principle to the sum of (\ref{eq:t-derivative}) and (\ref{eq:t-2nd-derivative}) yields
\begin{equation}
\label{2nd-upper-bound} 
\frac{\partial}{\partial t}\(\frac{\partial u}{\partial t}+u\)\leqslant \max\frac{\partial}{\partial t}\(\frac{\partial u}{\partial t}+u\)\vline_{t=0}=C,
\end{equation}
or equivalently, 
$$\frac{\p}{\p t}\(e^t\frac{\p u}{\p t}\)\leqslant Ce^t.$$
Thus, we conclude
$$\frac{\partial u}{\partial t}\leqslant C,$$ 
which gives an upper bound for the volume form ${\widetilde{\omega}_t}^n=e^{\frac{\partial u}{\partial t}}\Omega$ under the flow (\ref{eq:M-SKRF}).

As before, we now consider $u-Ct$ as the new $u$ such that 
$$\frac{\p u}{\p t}\leqslant 0$$ 
and so $u$, not just its spatial maximum, decreases along the flow. Thus, without loss of generality, we assume in the rest of this paper that 
\begin{equation}
\label{upper-bound}
u\leqslant 0, \qquad\frac{\p u}{\p t}\leqslant 0.
\end{equation}

Moreover, equation (\ref{eq:t-derivative}) can be reformulated as follows:
\begin{align*}
\frac{\p}{\p t}\(\frac{\p u}{\p t}+u\) & =\Delta\(\frac{\p u}{\p t}+u\)-n+\<\widetilde\omega_t, \omega_\infty\>+\frac{\p u}{\p t},\\
\frac{\p}{\p t}\(e^t\frac{\p u}{\p t}\) & =\Delta\(e^t\frac{\p u}{\p t}\)-\<\widetilde\omega_t,\omega_0-\omega_\infty\>+e^t\frac{\p u}{\p t}.
\end{align*}
The difference of these two equations is
\begin{align}
\label{eq:combo-1} 
\frac{\p}{\p t}\((e^t-1)\frac{\p u}{\p t}-u-nt\) & = \Delta\((e^t-1)\frac{\p u}{\p t}-u-nt\) \\
& \quad -\<\widetilde\omega_t,\omega_0\>+(e^t-1)\frac{\p u}{\p t}, \notag
\end{align}
and so 
$$\frac{\p}{\p t}\((e^t-1)\frac{\p u}{\p t}-u-nt\)< \Delta\((e^t-1)\frac{\p u}{\p t}-u-nt\)$$
since $\<\widetilde\omega_t,\omega_0\> > 0$ and $\frac{\p u}{\p t}\leqslant 0$. Then by the maximum principle, we have
$$(e^t-1)\frac{\p u}{\p t}-u-nt\leqslant \max\((e^t-1)\frac{\p u}{\p t}-u-nt\)\vline_{t=0}= 0,$$
or equivalently
\begin{equation}
\label{u_t-by-u}
(e^t-1)\frac{\p u}{\p t}-nt\leqslant u.
\end{equation}
Hence, any lower bound for $\frac{\p u}{\p t}$ implies a lower bound for $u$ at any finite time, which is also clear by the Fundamental Theorem of Calculus.   

\section{Characterisation of finite time singularities}\label{finite-time-sing}

In this section, we focus on the finite time (i.e. $0<T<\infty$) singularities of the modified K\"ahler-Ricci flow (\ref{eq:M-KRF}), or equivalently (\ref{eq:M-SKRF}). From general PDE theory, we know that $u$ or some of its spacetime derivatives will blow up, i.e. losing uniform bound in $[0, T)\times X$, which is equivalent to the loss of uniform bound for the flow metric, Riemannian curvature or its covariant derivatives. However, just as for the Ricci flow or the K\"ahler-Ricci flow, it is way more satisfying to characterise such singularity by the blow up of simple scalar geometric quantity, which is the goal of this section.      

Before diving into the details, we justify the existence of the flow limit in the sense of distribution which is comparable with the result in \cite{weak-limit-finite}. It lays the solid ground for further study of such singularities, for example, by providing the uniqueness of limit. 

Then, we give a geometric characterisation of the finite time singularity by showing the blow up of $\frac{\p u}{\p t}$ from below, i.e. degeneration of the volume form. The discussion is comparable with that in \cite{r-blow-up}, and leaves intriguing questions for further investigation.    

\subsection{Flow weak limit}\label{flow-weak-limit}

Here, we prove the existence of the flow weak limit in the finite time singularity case under a mild assumption. This prepares the ground for understanding the geometric behaviour of the flow towards the time of singularity. 

In this case, $[\omega_\infty]$ is outside the closure of the K\"ahler cone. Then the flow will develop singularity at $0<T<\infty$, and $\omega_T$ is on the boundary of the closure of the K\"ahler cone. In other words, this is the finite time singularity case. Since we have $\frac{\p u}{\p t}\leqslant 0$, by the elementary fact of pluripotential function \cite{lel69}, there exists a decreasing limit for $u$ in $PSH_{\omega_T}(X)$ as long as $u$ does not go to $-\infty$ uniformly.    

We first assume $[\omega_T]$ having a non-negative representative, i.e. there exists $f\in C^\infty(X)$ such that $\omega_T+\sqrt{-1}\p\bar\p f\geqslant 0$. In principle, this special case is weaker than $ [\omega_T]$ being semi-ample (i.e. a rational class with some multiple $k[\omega_T]$ generating a holomorphic map from $X$ to $\mathbb{CP}^N$ for some $N$), but not too different in practice.

\noindent {\bf Claim.} With the additional assumption, there is a (decreasing) flow weak limit $u_T\in PSH_{\omega_T}(X)\cap L^\infty(X)$ as $t\to T$ with the weak limit $\omega_T+\sqrt{-1}\p\bar\p u_T$ of $\widetilde\omega_t$ as $t\to T$. 

\begin{remark}

The key is the $L^\infty$ control, more precisely the uniform lower bound of the metric potential. The proof does not make use of the pluripotential theory argument as in \cite{demailly-pali}, \cite{ey-gu-ze} or \cite{zzo}. Instead, we apply purely flow techniques comparable to that in \cite{t-znote}. The weak convergence then follows in the same way as in \cite{modified}.

\end{remark}

In light of the general estimates (\ref{upper-bound}), we only need to search for the lower bound of $u$ for the $L^\infty$ control. Manipulating the variations of $t$-derivative of the flow equation in Section 2, we have 
\begin{align}
\label{eq:combo-2}
\frac{\p}{\p t}\((1-e^{t-T})\frac{\p u}{\p t}+u\) & = \Delta\((1-e^{t-T})\frac{\p u}{\p t}+u\)-n \\ & \quad +\<\widetilde\omega_t,
\omega_T\>+(1-e^{t-T})\frac{\p u}{\p t}. \notag
\end{align}

Let $F=(1-e^{t-T})\frac{\p u}{\p t}+u-f$. Then we have
\begin{equation}
\begin{split}
\frac{\p F}{\p t}&= \Delta F-n+\<\widetilde\omega_t,
\omega_T+\sqrt{-1}\p\bar\p f\>+F-u+f \\
&\geqslant \Delta F+F-C, \nonumber
\end{split}
\end{equation}
where we have used $\omega_T+\sqrt{-1}\p\bar\p f\geqslant 0$ and the upper bound of $u$. The maximum principle then gives $F\geqslant -C$ for $t\in [0, T)$.

The general upper bound for $\frac{\p u}{\p t}$ gives $F\leqslant u+C$, and so $u\geqslant -C$.

Since $u$ decreases along the flow, we conclude that as $t\to T$, the limit of $u$ satisfies $u_T\in PSH_{\omega_T}(X)\cap L^\infty(X)$. The weak convergence of $\widetilde\omega_t$ and its wedge powers (as positive currents) then follows by the classic result by Bedford and Taylor in \cite{bed-tay}.

Now we remove the assumption of $[\omega_T]$ having a non-negative representative and consider the finite time singularity case in general. $[\omega_T]$ is on the boundary of the K\"ahler cone for $X$, so there exists $\varphi\in PSH_{\omega_T}(X)$. From this fact, one can use the sequential limit construction in \cite{tian-survey}, more specifically, for the setting of the general flow discussed at the end of \cite{weak-limit-finite}. The fundamental regularisation result by Demailly (Theorem 1.6 in \cite{demailly-2015}) provides a decreasing approximation sequence $\{\varphi_m\}^\infty_{m=1}$ for $\varphi$ satisfying

\begin{itemize}

\item $\varphi_m\in PSH_{\omega_T+\frac{1}{m}\omega_0}(X)$;

\item $\varphi_m\in C^\infty(X\setminus Z_m)$ with $Z_m\subset Z_{m+1}$ being analytic subvarieties of $X$. Furthermore, $\varphi_m$ has logarithmic poles along $Z_m$, i.e. locally being the logarithm of a sum of squares of finitely many holomorphic functions, all vanishing along $Z_m$.

\end{itemize} 

We start with the following combination of (\ref{eq:combo-1}) and (\ref{eq:combo-2}):
\begin{equation} 
\begin{split}
& \frac{\partial}{\partial t}\(\frac{1}{m}[(1-e^t)\frac{\partial u}{\partial t}+u]+[(1-e^{t-T})\frac{\partial u}{\partial t}+u]\) \\
&\quad = \Delta\(\frac{1}{m}[(1-e^t)\frac{\partial u}{\partial t}+u]+[(1-e^{t-T})\frac{\partial u}{\partial t}+u]\) \\
&\qquad -\frac{n(1+m)}{m}+\<\widetilde\omega_t, \frac{1}{m}\omega_0+\omega_T\> \\
&\qquad +\(\frac{1}{m}(1-e^t)+(1-e^{t-T})\)\frac{\p u}{\p t}.
\end{split} \nonumber
\end{equation}
Let us modify this evolution equation using $\varphi_m$:  
\begin{equation} 
\label{eq:5}
\begin{split}
& \frac{\partial}{\partial t}\(\frac{1}{m}[(1-e^t)\frac{\partial u}{\partial t}+u]+[(1-e^{t-T})\frac{\partial u}{\partial t}+u]-\varphi_m\) \\
&\quad = \Delta\(\frac{1}{m}[(1-e^t)\frac{\partial u}{\partial t}+u]+[(1-e^{t-T})\frac{\partial u}{\partial t}+u]-\varphi_m\) \\
&\qquad  -\frac{n(1+m)}{m}+\<\widetilde\omega_t, \frac{1}{m}\omega_0+\omega_T+\sqrt{-1}\p\bar\p\varphi_m\> \\
&\qquad +\(\frac{1}{m}[(1-e^t)\frac{\partial u}{\partial t}+u]+[(1-e^{t-T})\frac{\partial u}{\partial t}+u]-\varphi_m\)-\frac{m+1}{m}u+\varphi_m,
\end{split}
\end{equation}
where $ \frac{1}{m}\omega_0+\omega_T+\sqrt{-1}\p\bar\p\varphi_m$ is smooth and positive over $X\setminus Z_m$. \\

\noindent{\bf Note:} We now assume $\varphi\geqslant -C$, and so $\varphi_m\geqslant -C$ for all $m$. In this case, one can use the simpler regularisation by B\l ocki-Ko\l odziej \cite{blo-koj}, where the $\varphi_m$ is smooth on $X$. The uniform upper bound is obvious. Of course, it would be very interesting to remove this additional assumption. \\

In light of the following inequality 
$$ -\frac{n(1+m)}{m}+\<\widetilde\omega_t, \frac{1}{m}\omega_0+\omega_T+\sqrt{-1}\p\bar\p\varphi_m\>-\frac{m+1}{m}u+\varphi_m\geqslant -C$$
which is uniform for $t\in [0, T)$ and $m$, we apply the maximum principle to conclude 
$$\frac{1}{m}[(1-e^t)\frac{\partial u}{\partial t}+u]+[(1-e^{t-T})\frac{\partial u}{\partial t}+u]-\varphi_m\geqslant -C,$$
which is uniform for all $m$'s over $X\times [0, T)$. Notice that the initial value of this expression clearly has a lower bound uniform for $m$'s by the uniform upper bound of $\varphi_m$'s. So we have over $X\times [0, T)$ and for all $m$'s, 
$$\(1+\frac{1}{m}\)u+\(\frac{1}{m}(1-e^t)+(1-e^{t-T})\)\frac{\p u}{\p t}\geqslant -C+\varphi_m.$$
For any fixed $t\in[0, T)$, we choose $m(t)$ large enough such that  
$$\frac{1}{m(t)}(1-e^t)+(1-e^{t-T})>0.$$
Combining with $u\leqslant 0$, $\frac{\p u}{\p t}\leqslant 0$ and $\varphi_{m(t)}\geqslant\varphi$, we conclude 
$$u\geqslant -C+\varphi,$$
which is enough to exclude the possibility of $u\to -\infty$ uniformly as $t\to T$. With the additional assumption of $\varphi\geqslant -C$, we have $u\geqslant -C$.

\subsection{Degeneration of volume form}\label{volume-form-degenerate}

We now prove that the volume form becomes degenerate in the case of $0<T<\infty$, in the sense that there cannot be a positive lower bound for the volume form towards the finite time of singularity. This and the result in Subsection \ref{flow-weak-limit} complete the proof of Theorem \ref{th:finite-time-singularity}.

We begin with the following inequality which is obtained by adjusting Lemma 5.8 in \cite{song-tian} to our modified K\"ahler-Ricci flow,
\begin{align*}
\(\frac{\p}{\p t}-\Delta\)\log\<\omega_0, \widetilde\omega_t\>
&\leqslant C\<\widetilde\omega_t, \omega_0\>+C +\frac{\<\omega_0, \Ric(\Omega)-e^{-t}(\omega_0-\omega_\infty)+\widetilde\omega_t\>}{\<\omega_0, \widetilde\omega_t\>}.
\end{align*}
For the adjustment, one just needs to realise that the difference between the evolution equations, namely (\ref{eq:M-SKRF}) and the more classic K\"ahler-Ricci flow as in \cite{song-tian}, only affects the $t$-derivative, which results in the last term on the right hand side. By the inequality, we have 
\begin{align}\label{eq:3.2-ineq1}
\(\frac{\p}{\p t}-\Delta\)\log\<\omega_0, \widetilde\omega_t\>
	\leqslant C\<\widetilde\omega_t, \omega_0\>+C+\frac{C}{\<\omega_0, \widetilde\omega_t\>}.
\end{align}
Recall the following equation,
$$\(\frac{\p}{\p t}-\Delta\)\((e^t-1)\frac{\p u}{\p t}-u\)=n+(e^t-1)\frac{\p u}{\p t}-\<\widetilde\omega_t, \omega_0\>.$$
Together with $\frac{\p u}{\p t}\leqslant 0$, we then have 
\begin{align}\label{eq:3.2-ineq2}
(e^t-1)\frac{\p u}{\p t}-u\leqslant C
\end{align}
for the finite time singularity case under study. 

Combining \eqref{eq:3.2-ineq1} and \eqref{eq:3.2-ineq2}, we arrive at
\begin{equation}
\begin{split}
&\(\frac{\p}{\p t}-\Delta\)\(\log\<\omega_0, \widetilde\omega_t\>+B\((e^t-1)\frac{\p u}{\p t}-u\)\) \\
&\quad \leqslant (C-B)\<\widetilde\omega_t, \omega_0\>+C\cdot B+\frac{C}{\<\omega_0, \widetilde\omega_t\>}, \nonumber
\end{split}
\end{equation}
where $B$ is a positive constant which will be fixed later. At the (local in time) maximum value point of this evolution quantity, i.e. $\log\<\omega_0, \widetilde\omega_t\>+B\bigl((e^t-1)\frac{\p u}{\p t}-u\bigr)$, if it is not at the initial time (otherwise we have full control), then we have
$$(C-B)\<\widetilde\omega_t, \omega_0\>+C\cdot B+\frac{C}{\<\omega_0, \widetilde\omega_t\>}\geqslant 0.$$
Recall the Cauchy-Schwarz inequality
$$\<\omega_0, \widetilde\omega_t\>\cdot\<\widetilde\omega_t, \omega_0\>\geqslant n^2,$$ 
and so we have $\frac{1}{\<\omega_0,\widetilde\omega_t\>}\leqslant\frac{\<\widetilde\omega_t, \omega_0\>}{n^2}$. Choosing $B$ sufficiently large, we have
$$\<\widetilde\omega_t, \omega_0\>\leqslant C.$$
Now notice the elementary inequality
$$\<\omega_0, \widetilde\omega_t\>\leqslant\<\widetilde\omega_t, \omega_0\>^{n-1}\cdot\frac{\widetilde\omega^n_t}
{\omega^n_0}.$$ 
Together with $\widetilde\omega^n_t=e^ {\frac{\p u}{\p t}}\Omega\leqslant C\Omega$, we arrive at 
$$\<\omega_0, \widetilde\omega_t\>\leqslant C$$
at the point under consideration. In light of the upper bound $(e^t-1)\frac{\p u}{\p t}-u\leqslant C$ (for $T<\infty$) and since we are looking at the point of maximal value, we conclude at that point and hence everywhere, 
$$\log\<\omega_0, \widetilde\omega_t\>+B\((e^t-1)\frac{\p u}{\p t}-u\)\leqslant C.$$
Using $u\leqslant 0$ and $\frac{\p u}{\p t}\leqslant 0$, we have 
$$\<\omega_0, \widetilde\omega_t\>\leqslant Ce^{-C\frac{\p u}{\p t}}.$$
Together with $\widetilde\omega^n_t=e^{\frac{\p u}{\p t}}\Omega$, we conclude 
$$\frac{1}{C}e^{C\frac{\p u}{\p t}}\omega_0\leqslant \widetilde\omega_t\leqslant Ce^{-C\frac{\p u}{\p t}}\omega_0.$$

Suppose that there is a positive lower bound for the volume form $\widetilde\omega^n_t$ uniform for $t\in [0, T)$, i.e.
$$\frac{\p u}{\p t}\geqslant -C.$$
Then we have the uniform bound for the flow metric 
$$\frac{\omega_0}{C}\leqslant\widetilde\omega_t\leqslant C\omega_0.$$
Now we can draw the contradiction as in \cite{r-blow-up} using the characterisation of K\"ahler class in \cite{demailly-paun}, which only needs the lower bound $\frac{\omega_0}{C}$. Hence, there cannot be a uniform lower bound for $\frac{\p u}{\p t}$, i.e. the volume form $\widetilde\omega^n_t=e^{\frac{\p 
u}{\p t}}\Omega$ has no uniform positive lower bound towards the singularity time $T<\infty$, which is the volume degeneration in Theorem \ref{th:finite-time-singularity}.

By the results in Subsections \ref{flow-weak-limit} and \ref{volume-form-degenerate}, Theorem \ref{th:finite-time-singularity} is proved.

Now we translate this non-existence of lower bound to a more precise geometric description along the flow. By $u\leqslant 0$ and the Fundamental Theorem of Calculus, $\frac{\p u}{\p t}\geqslant -C$ and $\frac{\p u}{\p t}+u\geqslant -C$ are equivalent for any finite time interval. So $\frac{\p u}{\p t}+u$ has no uniform lower bound either (for finite time singularity). Then in light of (\ref{2nd-upper-bound}), by considering $u-Ct$ instead of $u$ for an even larger $C$ if necessary without affecting the earlier arguments, we have
$$\frac{\p}{\p t}\(\frac{\p u}{\p t}+u\)\leqslant 0,$$
and so $\frac{\p u}{\p t}+u$ decreases along the flow. Then we know  
$$\min_{\{t\}\times X}\(\frac{\p u}{\p t}+u\)\to -\infty$$ 
as $t\to T<\infty$. Furthermore, by $u\leqslant 0$, $\frac{\p u}{\p t}\leqslant 0$ and (\ref{u_t-by-u}), we have
\begin{equation}
\label{equivalence}
e^T\frac{\p u}{\p t}-nT<e^t\frac{\p u}{\p t}-nt\leqslant \frac{\p u}{\p t}+u\leqslant \frac{\p u}{\p t}.
\end{equation} 
Hence we have $\min_{\{t\}\times X}\frac{\p u}{\p t}\to -\infty$, and so $\min_{\{t\}\times X}\frac{\widetilde\omega^n_t}{\Omega}=e^{\min_{\{t\}\times X}\frac{\p u}{\p t}}\to 0$ as $t\to T$, for a better geometric picture.

However, the attempt to search for a more pointwise statement runs into difficulties. In our setting, both $u$ and $\frac{\p u}{\p t}+u$ decrease along the flow, so let's set 
$$\lim_{t\to T} u=U, ~~~~\lim_{t\to T} \(\frac{\p u}{\p t}+u\)=V$$
with $U$ and $V$ both valued in $[-\infty, 0]$ with $-\infty$ as a possible value. At points where neither $U$ nor $V$ is $-\infty$,   
$$\lim_{t\to T} \frac{\p u}{\p t}=V-U.$$
For points where at least one function takes $-\infty$, we need to be more careful. Clearly, $\{U=-\infty\}\subset \{V=-\infty\}$ as $u\leqslant 0$, and 

\begin{itemize}

\item If $V=-\infty$ and $U>-\infty$, then $\lim_{t\to T} \frac{\p u}{\p t}=-\infty$. 

\item If $U=-\infty$ (and so $V=-\infty$), then by (\ref{u_t-by-u}), $\lim_{t\to T} \frac{\p u}{\p t}=-\infty$. 

\end{itemize}

Alternatively, it is clear from (\ref{equivalence}) that $\lim_{t\to T} \frac{\p u}{\p t}=-\infty$ exactly when it is the case for $\frac{\p u}{\p t}+u$, i.e. $V=-\infty$. Thus we conclude
$$\Big\{\frac{\p u}{\p t}+u\to -\infty\Big\}=\Big\{\frac{\p u}{\p t}\to-\infty\Big\}.$$ 
As a decreasing limit of smooth functions, we only know $V$ is upper semi-continuous, i.e. for any $p\in X$,
$$\lim_{x \to p}\sup V(x)=V(p)$$ 
with ``$=$'' replaced by ``$\leqslant$'' if we consider essential supreme instead. Although we know from the discussion $\inf_X V=-\infty$, it is unclear whether it indeed takes $-\infty$ somewhere. So it is possible for these two sets to be empty. If $U$ takes $-\infty$ somewhere, then it is also the case for $V$. However, for many interesting cases, $u$ is actually bounded, and so is $U$. Of course, the sets can also be large. For example, when $[\omega_T]^n=0$,  
$$\int_X e^V\Omega=\lim_{t\to T}\int_X e^{u+\frac{\p u}{\p t}}\Omega\leqslant \lim_{t\to T}\int_X e^{\frac{\p u}{\p t}}\Omega=\lim_{t\to T}[\omega_t]^n=[\omega_T]^n=0.$$ 
and so $V=-\infty$ almost everywhere.   

\begin{remark}

Although we make a convenient choice of $u$ by adjusting it as $u-Ct$ for some possibly large $C$, the above conclusions are obviously independent of such choice.  

\end{remark}

\section{Geometric limit}\label{geom-lim}

In this section, we seek to control the metric geometry along the modified K\"ahler-Ricci flow for the special degenerate case called canonical surgical contraction as discussed in \cite{song-weinkove, song-weinkove-2} for the classic K\"ahler-Ricci flow. 

More precisely, the singularity happens at $T\in [0, \infty]$ with $[\omega_T]$ (of course nef.) semi-ample and big by assumption. In effect, the semi-ampleness assumption gives $m[\omega_T]=F^*[\omega]$ for a positive integer $m$ and a holomorphic map 
$$F: X\to \mathbb{CP}^N,$$ 
where $\omega$ is the Fubini-Study metric and the image $F(X)$, which is singular in general, has the same dimension as $X$ which indicates the non-collapsed property, i.e. $[\omega_T]^n>0$. For our purpose, we make proper choices so that $m\omega_T=F^*\omega$. 

In general, the subset of $X$, where $dF$ is not injective, is a subvariety of $X$. It is contained in a divisor of $X$, at least when $X$ is projective. We consider such a divisor $E=\{\sigma=0\}$, where $\sigma$ is the defining section of the corresponding holomorphic line bundle and $|\cdot|$ is a Hermitian metric. Then there exists a constant $\beta>0$ such that 
$$C|\sigma|^{2\beta}\omega_0\leqslant F^*\omega\leqslant C\omega_0.$$ 
Moreover, $[\omega_T]-\epsilon E$ is K\"ahler for sufficiently small $\epsilon>0$, which essentially comes from the bigness assumption. 

For a canonical surgical contraction, if $T<\infty$, $F$ blows down $\mathbb{CP}^{n-1}$ to smooth point; if $T=\infty$, $F$ blows down chains of $\mathbb{CP}^{n-1}$ to isolated singular points. Intuitively speaking, they are the counterpart of blowing down $(-1)$ and $(-2)$-curves for complex surfaces in general dimensions. 

The main estimate of this section is the following.

\begin{theorem}
\label{th:improved-metric-estimate} In the above setting, if we further assume that for $[0, T)\times X$, there is a constant $C$ satisfying 
\begin{equation}
\label{assumption} 
\Ric(\Omega)-e^{-t}(\omega_0-\omega_\infty)\leqslant CF^*\omega, 
\end{equation} 
then under the modified K\"ahler-Ricci flow \eqref{eq:M-KRF}, $\widetilde\omega_t$ satisfies 
\begin{enumerate}

\item[a)] $\widetilde\omega_t\leqslant C|\sigma|^{-2\beta}F^* \omega$,

\item[b)] $\widetilde\omega_t\leqslant C(A)|\sigma|^{-\frac{2\beta A}{A+1}}\omega_0$ for a sufficiently large $A>0$.
 
\end{enumerate}
\end{theorem}

In general, we only care about the flow estimates as $t\to T^-$, which would help to elucidate the essence of (\ref{assumption}). We justify (\ref{assumption}) for the following interesting cases:

\begin{itemize}
	
	\item[(1)] $K_X=-c_1(X)$ semi-ample:
	
	In this case, we can take $\Ric(\Omega)\leqslant 0$. If we take $\omega_0-\omega_\infty\geqslant 0$, i.e. $\omega_0$ sufficiently positive, then 
	$$\Ric(\Omega)-e^{-t}(\omega_0-\omega_\infty)\leqslant 0$$
	which is (\ref{assumption}) for $C=0$. This is exactly the setting of Theorem \ref{th:metric-control}. 
	
	The case of $X$ being Calabi-Yau, i.e. $c_1(X)=0$ is included here. There have been studies of such degeneration using the elliptic setting, e.g. \cite{tosatti, gross-tosatti-zhang20}. 
	
	\item[(2)] $\omega_\infty=-\Ric(\Omega)$:
	
	In this case, (\ref{assumption}) reads
	$$-\omega_t=-\omega_\infty-e^{-t}(\omega_0-\omega_\infty)\leqslant CF^*\omega.$$
	For $t\in [0, T)$, $\omega_t=\alpha\omega_0+(1-\alpha)\omega_T\geqslant 0$ for $\alpha\in (0, 1]$ and $\omega_T=\frac{1}{m}F^*\omega\geqslant 0$. So (\ref{assumption}) is true for $C=0$. 
	
\end{itemize}

\begin{proof}[Proof of Theorem \ref{th:improved-metric-estimate}]

We begin by setting the normalisation of $u$. Let $v=u-\frac{\int_X u\Omega}{\int_X\Omega}$. 

In light of $m\omega_T=F^*\omega$, if $T<\infty$, then $u$ is bounded by the discussion in the beginning of Section 3, and so $|v|\leqslant C$; if $T=\infty$, then the upper bound of $v$ follows from standard argument involving Green's function, while the lower bound can be obtained in light of $\frac{\p u}{\p t}\leqslant 0$ and the generalisation of Kolodziej's result from pluripotential theory in \cite{koj98} to the semi-ample and big case in \cite{demailly-pali}, \cite{ey-gu-ze} and \cite{zzo} \footnote{In fact, we know $|u|\leqslant C$ here.}. Moreover since $\frac{\p u}{\p t}\leqslant 0$,  
$$\frac{\p v}{\p t}=\frac{\p u}{\p t}-\frac{\int_X\frac{\p u}{\p t}\Omega}{\int_X\Omega}\geqslant\frac{\p u}{\p t}.$$ 

Now we make use of the computations in \cite{song-tian} and \cite{song-weinkove}. It is straightforward to compute the following two equations:
\begin{align*}
\(\frac{\p}{\p t}-\Delta\)v & =\frac{\p v}{\p t}-n+\<\widetilde \omega_t, \omega_t\>,\\
\(\frac{\p}{\p t}-\Delta\)\log |\sigma|^2 & =\<\widetilde\omega_t, -\sqrt{-1}\p\bar\p\log |\sigma|^2\>.
\end{align*}
The following two inequalities follow from those in \cite{song-tian} and \cite{song-weinkove} with slight modifications as described at the beginning of Subsection 3.2:  
\begin{align}
\(\frac{\p}{\p t}-\Delta\)\log\<\omega_0, \widetilde\omega_t\> & \leqslant C\<\widetilde\omega_t, \omega_0\>+C+\frac{\<\omega_0, \Ric(\Omega)-e^{-t}(\omega_0-\omega_\infty)\>}{\<\omega_0, \widetilde\omega_t\>}, \label{ineq1}\\
\(\frac{\p}{\p t}-\Delta\)\log\<F^*\omega,\widetilde\omega_t\> & \leqslant C\<\widetilde\omega_t, F^*\omega\>+C+\frac{\<F^*\omega, \Ric(\Omega)-e^{-t}(\omega_0-\omega_\infty)\>}{\<F^*\omega, \widetilde\omega_t\>}, \label{ineq2}
\end{align}
where the second one is only considered where $F^*\omega>0$. The inequalities are obtained by pointwise calculation as for the classic Schwarz Lemma with the constant $C$ depending on the bisectional curvature of $\omega_0$ and $F^*\omega$, respectively. For $F^*\omega$ which is not a metric globally on $X$, $C$ is just for the bisectional curvature of $\omega$.

In light of $\<\omega_0, \Ric(\Omega)-e^{-t}(\omega_0-\omega_\infty)\>\leqslant C$ and $\frac{1}{\<\omega_0, \widetilde\omega_t\>}\leqslant\frac{\<\widetilde\omega_t, \omega_0\>}{n^2}$, \eqref{ineq1} gives
$$\(\frac{\p}{\p t}-\Delta\)\log\<\omega_0, \widetilde\omega_t\>\leqslant C\<\widetilde\omega_t, \omega_0\>
+C.$$
For \eqref{ineq2}, we use the assumption
$$\Ric(\Omega)-e^{-t}(\omega_0-\omega_\infty)\leqslant CF^*\omega$$ 
and arrive at
$$\(\frac{\p}{\p t}-\Delta\)\log\<F^*\omega, \widetilde\omega_t\>\leqslant C\<\widetilde\omega_t, F^*\omega\>
+C.$$
Now we apply the scheme as in \cite{song-weinkove} to control the flow metric. More precisely, we study the following quantity $$Q_\delta=\log\<\omega_0, \widetilde\omega_t\>+A\log\(|\sigma|^{2\beta}\<F^*\omega, \widetilde\omega_t\>\)+A\delta \log|\sigma|^2-ABv,$$ 
for positive constants $A$ and $B$ which will be fixed later and $\delta\in (0, 1]$. In light of $|\sigma|^{2\beta}C\omega_
0\leqslant F^*\omega$, we have 
$$|\sigma|^{2\beta}\<F^*\omega, \widetilde\omega_t\>\leqslant C\<\omega_0, \widetilde\omega_t\>.$$ 
So it makes sense to talk about the maximum value of $Q_\delta$ in $X\times [0, T_0]$ for any $T_0<T$, which is achieved in $X\setminus\{\sigma= 0\}$ where $Q_\delta$ is smooth. 

Recall our notation that $C$ is a positive constant that can change from line to line, and positive constants depending on $A$ or $B$ are denoted by $C(A)$, $C(B)$ or $C(A, B)$.

The previous calculations give the following control of the evolution of $Q_\delta$, 
\begin{equation}
\begin{split}
\(\frac{\p}{\p t}-\Delta\)Q_\delta&\leqslant C\<\widetilde\omega_t, \omega_0\> +A\<\widetilde\omega_t, C\omega_T-(\beta+\delta)\sqrt{-1}\p\bar\p\log |\sigma|^2-B\omega_t\> \\
&~~~~ -AB\frac{\p v}{\p t}+C(A, B). \nonumber
\end{split}
\end{equation}
For $t\in [0, T]$, $\omega_t=(1-\gamma)\omega_0+\gamma\omega_T$ with $\gamma\in [0, 1]$. As $\omega_0\geqslant C\omega_T$, $$\omega_t\geqslant (C-C\gamma+\gamma)\omega_T\geqslant \min\{C, 1\}\omega_T=C\omega_T.$$
So it follows that
$$C\omega_T-(\beta+\delta)\sqrt{-1}\p\bar\p\log |\sigma|^2-B\omega_t \leqslant (-CB+C)\omega_T-(\beta+\delta)\sqrt{-1}\p\bar\p\log |\sigma|^2.$$
Choosing $B$ sufficiently large, we have $-CB+C$ sufficiently negative so that for a properly chosen $|\cdot|$,
$$(-CB+C)\omega_T-(\beta+\delta)\sqrt{-1}\p\bar\p\log |\sigma|^2\leqslant -C(B) \omega_0$$ 
since $[\omega_T]-\epsilon E$ is K\"ahler for $\epsilon>0$ sufficiently small. Combining with $\frac{\p v}{\p t}\geqslant \frac{\p u}{\p t}$, we then obtain
$$\(\frac{\p}{\p t}-\Delta\)Q_\delta\leqslant C\<\widetilde\omega_t, \omega_0\>-A\< \widetilde\omega_t,
C(B)\omega_0\>-AB\frac{\p u}{\p t}+C (A, B).$$ 
Choosing $A$ large enough such that $C-AC(B)<-1$, then we get 
$$\(\frac{\p}{\p t}-\Delta\)Q_\delta\leqslant -\<\widetilde\omega_t, \omega_0\>-AB\frac{\p u}{\p t}+C (A, B).$$ 
At the maximum value point of $Q_\delta$, we have 
$$\<\widetilde\omega_t, \omega_0\>\leqslant C(A, B)-AB\frac{\p u}{\p t}.$$
By elementary inequality and (\ref{eq:M-SKRF}), 
$$\<\widetilde\omega_t, \omega_0\>\geqslant n\(\frac{\omega_0^n}{\widetilde\omega_t^n}\)^{\frac{1}{n}}\geqslant Ce^{-\frac{1}{n}\frac{\p u}{\p t}}.$$
Thus at the point of interest, 
$$Ce^{-\frac{1}{n}\frac{\p u}{\p t}}\leqslant C(A, B)-AB\frac{\p u}{\p t},$$
which gives $\frac{\p u}{\p t}\geqslant -C(A, B)$, and so
$$\<\widetilde\omega_t, \omega_0\>\leqslant C(A, B).$$
In light of the elementary inequality 
$$\<\omega_0, \widetilde\omega_t\>\leqslant \<\widetilde\omega_t, \omega_0\>^{n-1}\cdot\frac{\widetilde\omega_t^n}{\omega_0^n}$$ 
and the volume upper bound for $\widetilde\omega_t^n$, we have at that point,
$$\<\omega_0, \widetilde\omega_t\>\leqslant C(A, B).$$
So we have the uniform upper bound for $Q_\delta$ by $|\sigma|^{2\beta}\<F^*\omega,\widetilde\omega_t\>\leqslant C\<\omega_0, \widetilde\omega_t\>$ and the bound for $v$. Hence we conclude
$$\log\<\omega_0, \widetilde\omega_t\>+A\log\(|\sigma|^{2(\beta+\delta)}\<\omega_T, \widetilde\omega_t\>\)\leqslant C(A, B),$$
where the constant $C$ on the right hand side is independent of $\delta$ as is clear from the argument. Taking $\delta\to 0^+$, we arrive at 
\begin{equation}
\label{main-est}
\log\<\omega_0, \widetilde\omega_t\>+A\log\(|\sigma|^{2\beta}\<F^*\omega, \widetilde\omega_t\>\)\leqslant C(A, B).
\end{equation}
Since $\omega_0\leqslant C|\sigma|^{-2\beta}F^*\omega$, we have $\<\omega_0, \widetilde\omega_t\>\geqslant C|\sigma|^{2\beta}\<F^*\omega, \widetilde\omega_t\>$. Then (\ref{main-est}) gives $|\sigma|^{2\beta}\<F^*\omega, \widetilde\omega_t\>\leqslant C(A, B)$, which implies Part $a)$ in Theorem \ref{th:improved-metric-estimate}:    
$$\widetilde\omega_t\leqslant C(A, B)|\sigma|^{-2\beta}F^* \omega.$$ 
Meanwhile by $\omega_0\geqslant CF^*\omega$, (\ref{main-est}) also implies
$$(1+A)\log\<\omega_0, \widetilde\omega _t\>+\beta A\log|\sigma|^2\leqslant C(A, B),$$ 
and so $\<\omega_0, \widetilde\omega_t\>\leqslant C(A, B)|\sigma|^ {-\frac{2\beta A}{A+1}}$. Thus we conclude
$$\widetilde\omega_t\leqslant C(A, B)|\sigma|^ {-\frac{2\beta A}{A+1}}\omega_0.$$
Notice that we can fix a sufficiently large $B$ first, and then $A$ just needs to be sufficiently large, thus proving Part b) of Theorem \ref{th:improved-metric-estimate}.

Therefore, Theorem \ref{th:improved-metric-estimate} is proved.
\end{proof}

We now apply Theorem \ref{th:improved-metric-estimate} to prove Theorem \ref{th:metric-control}.

\begin{proof}[Proof of Theorem \ref{th:metric-control}] When $T=\infty$ and $\Ric(\Omega)\leqslant 0$, by taking $\omega_0-\omega_\infty\geqslant 0$, i.e. $\omega_0$ sufficiently positive,  
$$\Ric(\Omega)-e^{-t}(\omega_0-\omega_\infty)\leqslant 0.$$
So in the setting of Theorem \ref{th:metric-control}, we can apply Theorem \ref{th:improved-metric-estimate}. Then we can apply the arguments in Sections 2 and 3 of \cite{song-weinkove} and Sections 2--5 of \cite{song-weinkove-2}. 

For the canonical surgical contraction map $F$, $\beta=1$. We know from Part a) of Theorem \ref{th:improved-metric-estimate} that the flow metric is uniformly bounded in the directions of the contracted divisor. Meanwhile, since $\frac{2\beta A}{A+1}<2\beta$, Part b) of Theorem \ref{th:improved-metric-estimate} provides the control in the transverse direction and gives the uniform bound of distance to the contracted divisor. 

Since the divisor is $\mathbb{CP}^{n-1}$ and $\omega_\infty=F^*\omega$, we can justify that the contracted divisor shrinks metrically to a point. Finally, one has the Gromov-Hausdorff convergence of the flow metric to the metric over $F(X)$ which is known to be smooth away from $F(E)$ by \cite{t-znote}. This concludes the proof of Theorem \ref{th:metric-control}. 
\end{proof}

\section{Degenerate complex Monge-Amp\`ere equation}\label{deg-MA}

A major motivation to study the modified K\"ahler-Ricci flow (\ref{eq:M-KRF}) is the limiting equation (\ref{eq:D-CY}), the renowned Calabi-Yau equation with degenerate (i.e. non-K\"ahler) class $[\omega_\infty]$. With that goal, we focus on the case of $T=\infty$, i.e. when the flow develops an infinite time singularity with the limit satisfying the (degenerate) complex Monge-Amp\`ere equation in a proper sense. Here, we assume $[\omega_\infty]$ to be semi-ample, i.e. $\omega_\infty=F^*\omega$ as in Section \ref{geom-lim}, but now we no longer require $F(X)$ to be of the same dimension as $X$. Thus, we have $[\omega_\infty]^n\geqslant 0$.  

The case of $[\omega_\infty]^n>0$ is non-collapsed. For a Calabi-Yau manifold $X$, i.e. $X$ K\"ahler with $c_1(X)=0$, we take $\Omega$ with $\Ric(\Omega)=0$ and $\omega_0$ such that $\omega_0-\omega_\infty\geqslant 0$, and the flow metric of (\ref{eq:M-KRF}) converges smoothly out of the stable base locus set of $[\omega_\infty]$ by \cite{modified, yuan-convergence} to the unique solution of bounded metric potential \cite{demailly-pali, ey-gu-ze, zzo, dinew-z}. Furthermore, if $F$ is a canonical surgical contraction, then Theorem \ref{th:metric-control} gives the global convergence in the Gromov-Hausdorff sense. This corresponds to the study using the elliptic continuity method by \cite{tosatti}. 

The case of $[\omega_\infty]^n=0$ is collapsed, so the dimension of $F(X)$ (singular in general) is strictly smaller than the dimension of $X$, and $F^*\omega$ is nowhere a metric over $X$. If $X$ is a Calabi-Yau manifold, the generic fibre of $F$ is Calabi-Yau by elementary consideration. In this setting, it is more convenient to consider the following scaling of the modified K\"ahler-Ricci flow \eqref{eq:M-SKRF}, or equivalently \eqref{eq:M-KRF}, with a generic fibre of dimension $r$: 
\begin{equation}
\label{eq:M-SSKRF}
\frac{\partial v}{\partial t}={\rm log}\frac{\widetilde\omega_t^n}{e^{-rt}\Omega}={\rm log}\frac{(\omega_t+\sqrt{-1}\partial\bar{\partial}v)^n}{e^{-rt}\Omega}, \qquad v(0, \cdot)=0,
\end{equation}
where 
$$u=v-\frac{r}{2}t^2$$ and $$\widetilde\omega_t=\omega_t+\sqrt{-1}\p\bar\p u=\omega_t+\sqrt{-1}\p\bar\p v.$$ 
For the fibration $F$, we know $\omega_t^n\leqslant Ce^{-rt}\Omega$. Applying the maximum principle to equation \eqref{eq:M-SSKRF} for $v$, we have
$$v\leqslant Ct,$$
and so $u\leqslant -Ct^2$. Since $\frac{\p u}{\p t}\leqslant C$, we have
$$\frac{\p v}{\p t}=\frac{\p u}{\p t}+rt\leqslant rt+C.$$ 

In the following, we improve these bounds by comparing with the flow
\begin{equation}
\label{eq:KRF} 
\frac{\partial\widehat{\omega}_t}{\partial t}=-{\rm Ric}(\widehat{\omega}_t)+{\rm Ric}(\Omega)-\widehat\omega_t+\omega_\infty, \qquad \widehat{\omega}_0=\omega_0,
\end{equation}
where $\widehat\omega_t=\omega_t+\sqrt{-1}\p\bar\p w$ with $w$ satisfying
\begin{equation}
\label{eq:SKRF}
\frac{\partial w}{\partial t}={\rm log}\frac{\widehat\omega_t^n}{e^{-rt}\Omega}-w={\rm log}\frac{(\omega_t+\sqrt{-1}\partial\bar{\partial}w)^n}{e^{-rt}\Omega}-w, \qquad w(0, \cdot)=0.
\end{equation}
We have the following estimates from \cite{fong-z, song-tian-scalar}. A proof is summarised in the Appendix for completeness. 
\begin{prop}
\label{KRF-control}
Consider $\omega_\infty=F^*\omega$ for a map $F: X\to Y$ with possibly singular image $F(X)$, the generic fibre of dimension $r$ and a K\"ahler metric $\omega$ over $Y$. The flow (\ref{eq:KRF}), or equivalently (\ref{eq:SKRF}), exists for all $t\geqslant 0$, and we have the following uniform estimates for all time
$$|w|\leqslant C, \qquad \vline\frac{\p w}{\p t}\vline\leqslant C.$$
\end{prop}

Taking the difference of equations (\ref{eq:M-SSKRF}) and (\ref{eq:SKRF}), we get
\begin{equation}
\frac{\partial (v-w)}{\partial t}={\rm log}\frac{\bigl(\omega_t+\sqrt{-1}\partial\bar{\partial}w+\sqrt{-1}\partial\bar{\partial}(v-w)\bigr)^n}{(\omega_t+\sqrt{-1}\partial\bar{\partial}w)^n}+w, \qquad (v-w)(0, \cdot)=0. \nonumber
\end{equation}
Applying the maximum principle to $v-w$ with $|w|\leqslant C$, we arrive at
$$|v-w|\leqslant Ct,$$
and so
$$-C-Ct\leqslant v\leqslant C+Ct$$
where the upper bound can be chosen from this one or $v\leqslant Ct$ derived earlier. Replacing $\Omega$ by $e^{-C}\Omega$ and using the initial value $C$ for (\ref{eq:M-SSKRF}), and consider $v+Ct+C$ as the new $v$, then we have
$$0\leqslant v\leqslant C+At,$$ 
where we use $A$ to specify the coefficient of the linear upper bound.

The time derivative of (\ref{eq:M-SSKRF}) is
$$\frac{\p}{\p t}\(\frac{\p v}{\p t}\)=\Delta\(\frac{\p v}{\p t}\)-e^{-t}\<\widetilde\omega_t, \omega_0-\omega_\infty\>+r,$$
and we also have
$$\frac{\p}{\p t}\(\frac{\p v}{\p t}+v\)=\Delta\(\frac{\p v}{\p t}+v\)-n+r+\<\widetilde\omega_t, \omega_\infty\>+\frac{\p v}{\p t}.$$

For any constant $S$, we manipulate the two equations above to get  
$$\frac{\p}{\p t}\((1-e^S)\frac{\p v}{\p t}+v\)=\Delta\((1-e^S)\frac{\p v}{\p t}+v\)-n+r(1-e^S)+\<\widetilde\omega_t, \omega_{t-S}\>+\frac{\p v}{\p t}.$$

For $S<0$, we have
\begin{equation}
\begin{split}
&~~~~ \frac{\p}{\p t}\((1-e^S)\frac{\p v}{\p t}+v-w(t-S, \cdot)\) \\
&= \Delta\((1-e^S)\frac{\p v}{\p t}+v-w(t-S, \cdot)\)-n+r(1-e^S)+\<\widetilde\omega_t, \widehat\omega_{t-S}\>+\frac{\p v}{\p t}-\frac{\p w(t-S, \cdot)}{\p t} \\ 
&\geqslant \Delta\((1-e^S)\frac{\p v}{\p t}+v-w(t-S, \cdot)\)-C+n\(\frac{\widehat\omega_{t-S}^n}{\widetilde\omega_t^n}\)^{1/n}+\frac{\p v}{\p t}-\frac{\p w(t-S, \cdot)}{\p t} \\
&\geqslant \Delta\((1-e^S)\frac{\p v}{\p t}+v-w(t-S, \cdot)\)-C+Ce^{-\frac{1}{n}\frac{\p v}{\p t}}+\frac{\p v}{\p t},
\end{split} \nonumber
\end{equation}
where the last step makes use of Proposition \ref{KRF-control}. At the spacetime minimal value point of $(1-e^S)\frac{\p v}{\p t}+v-w(t-S, \cdot)$, if it is not at the initial time $t=0$, then 
$$0\geqslant -C+Ce^{-\frac{1}{n}\frac{\p v}{\p t}}+\frac{\p v}{\p t},$$
and so $\frac{\p v}{\p t}\geqslant -C$. In light of $v\geqslant 0$ and $|w|\leqslant C$, we conclude  
$$(1-e^S)\frac{\p v}{\p t}+v-w(t-S, \cdot)\geqslant -C.$$ 
Since $S<0$, $v\leqslant C+At$ and $|w|\leqslant C$, we arrive at 
$$\frac{\p v}{\p t}\geqslant -\frac{A}{1-e^S}t-C(S)$$
where $C(S)$ is a positive constant depending on $S$. Renaming $S$ by $-S$, then we have derived the lower bound
\begin{align}\label{lbound}
\frac{\p v}{\p t}\geqslant -\frac{A}{1-e^{-S}}t-C(S).
\end{align}
for any $S>0$.

For $S>0$, we proceed differently as follows.
\begin{equation}
\begin{split}
&~~~~ \frac{\p}{\p t}\((1-e^S)\frac{\p v}{\p t}+v-w(t-S, \cdot)\) \\
&= \Delta\((1-e^S)\frac{\p v}{\p t}+v-w(t-S, \cdot)\)-n+r(1-e^S)+\<\widetilde\omega_t, \widehat\omega_{t-S}\>+\frac{\p v}{\p t}-\frac{\p w(t-S, \cdot)}{\p t} \\ 
&\geqslant \Delta\((1-e^S)\frac{\p v}{\p t}+v-w(t-S, \cdot)\)-C-\frac{1}{e^S-1}\((1-e^S)\frac{\p v}{\p t}+v-w(t-S, \cdot)\)
\end{split} \nonumber
\end{equation}
where we've used $v\geqslant 0$ and $|w|\leqslant C$ in the last step. For the space minimum $\alpha(t)$ of $(1-e^S)\frac{\p v}{\p t}+v-w(t-S, \cdot)$ (of course for $t\geqslant S$), we have
$$\frac{d \alpha}{d t}\geqslant -C\alpha-C,$$
and so $\alpha\geqslant -C$ with these $C$'s depending on $S$. Then we have 
$$(1-e^S)\frac{\p v}{\p t}+v-w(t-S,\cdot) \geqslant -C(S).$$
As $S>0$, $v\leqslant C+At$ and $|w|\leqslant C$, we have derived the lower bound
\begin{align}\label{ubound}
\frac{\p v}{\p t}\leqslant \frac{A}{e^S-1}t+C(S)
\end{align}
with the coefficient of the linear upper bound as close to $0$ as we want by choosing $S$ sufficiently large. The constant $C(S)$ can also take care of $t\in [0, S)$. 

Let us summarise the estimates \eqref{lbound} and \eqref{ubound} as follows.

\begin{theorem}\label{th:collapsed-volume-v}
In the collapsed case, for the modified K\"ahler-Ricci flow (\ref{eq:M-SSKRF}), by properly choosing $\Omega$, we have positive constants $C$ and $A$ such that for any $S>0$,
\begin{align*}
0 & \leqslant v\leqslant C+At, \\
-\frac{A}{1-e^{-S}}t-C(S) & \leqslant \frac{\p v}{\p t}\leqslant \frac{A}{e^S-1}t+C(S),
\end{align*}
where the constant $C(S)$ depending on $S$. 
\end{theorem} 

\begin{proof}[Proof of Theorem \ref{th:collapsed-volume}] Theorem \ref{th:collapsed-volume} follows from Theorem \ref{th:collapsed-volume-v} in light of $u=v-\frac{r}{2}t^2$.
\end{proof}

\section{Appendix}\label{appendix}

Proposition \ref{KRF-control} is proved using essentially the same argument as in \cite{fong-z, song-tian-scalar}, though the flow under consideration has changed. We include the detail for the readers' convenience. We also point out that the Proposition holds for the setting way more general than $F$ being a fibre bundle, see for example Section 6 of \cite{fong-z}.  

\begin{proof}[Proof of Proposition \ref{KRF-control}]
In the setting of the proposition, $\omega_t^n$ is only bounded from above by some constant multiple of $e^{-rt}\Omega$. So applying the maximum principle to (\ref{eq:SKRF}) only gives
$$w\leqslant C.$$ 
Taking $t$-derivative for (\ref{eq:SKRF}), we get
$$\frac{\p}{\p t}\(\frac{\p w}{\p t}\)=\Delta\(\frac{\p w}{\p t}\)-e^{-t}\<\widehat\omega_t, \omega_0-\omega_\infty\>-\frac{\p w}{\p t}+r,$$
where $\Delta$ is with respect to the flow metric $\widehat\omega_t$. The above equation can be transformed to 
$$\frac{\p}{\p t}\(\frac{\p w}{\p t}+w\)=\Delta\(\frac{\p w}{\p t}+w\)-n+r+\<\widehat\omega_t, \omega_\infty\>,$$
$$\frac{\p}{\p t}\(e^t\frac{\p w}{\p t}\)=\Delta\(\frac{\p w}{\p t}\)-\<\widehat\omega_t, \omega_0-\omega_\infty\>+re^t.$$
The difference of the these two equations is
$$\frac{\p}{\p t}\((e^t-1)\frac{\p w}{\p t}-w\)=\Delta\((e^t-1)\frac{\p w}{\p t}-w\)+n-r+re^t-\<\widehat\omega_t, \omega_0\>.$$
It then follows that
$$\frac{\p}{\p t}\((e^t-1)\frac{\p w}{\p t}-w-(n-r)t-re^t\)\leqslant \Delta\((e^t-1)\frac{\p w}{\p t}-w-(n-r)t-re^t\).$$
By the maximum principle, we then get
$$(e^t-1)\frac{\p w}{\p t}-w-(n-r)t-re^t\leqslant C,$$
and so 
$$\frac{\p w}{\p t}\leqslant \frac{w+(n-r)t+re^t+C}{e^t-1},$$
which gives $\frac{\p w}{\p t}\leqslant C$ in light of the uniform upper bound for $w$ and the obvious case for short time. Then the result from pluripotential theory in \cite{demailly-pali, ey-gu-ze-collapsed} can be applied to (\ref{eq:SKRF}) to give 
$$|w|\leqslant C.$$ 

Now we recycle the proof of Proposition 2.1 in \cite{song-tian-scalar}. By patching up solutions for the corresponding complex Monge-Amp\`ere equations, we end up with a spacetime function $\psi$ such that 
$$\frac{1}{C}e^{-rt}\Omega\leqslant (\omega_t+\sqrt{-1}\p\bar\p \psi)^n\leqslant Ce^{-rt}\Omega, \quad |\psi|\leqslant C, \quad\text{and}\quad \vline\frac{\p \psi}{\p t}\vline\leqslant C.$$ 
Now adding up the following two inequalities
$$\frac{\p}{\p t}\(\frac{\p w}{\p t}+w\)\geqslant\Delta\(\frac{\p w}{\p t}+w\)-n+r,$$
\begin{equation}
\begin{split}
\frac{\p (w-\psi)}{\p t}
&= \Delta (w-\psi)-n+\<\widehat\omega_t, \omega_t+\sqrt{-1}\p\bar\p \psi\>+\frac{\p w}{\p t}-\frac{\p \psi}{\p t} \\
&\geqslant \Delta (w-\psi)-n+n\(\frac{(\omega_t+\sqrt{-1}\p\bar\p \psi)^n}{\widehat\omega_t^n}\)^{1/n}+\frac{\p w}{\p t} \\
&\geqslant \Delta (w-\psi)-n+\frac{\p w}{\p t}+C\(\frac{e^{-rt}\Omega}{e^{\frac{\p w}{\p t}+w}e^{-rt}\Omega}\)^{1/n} \\
&\geqslant \Delta (w-\psi)-n+\frac{\p w}{\p t}+Ce^{-\frac{1}{n}\frac{\p w}{\p t}},
\end{split} \nonumber
\end{equation}
we then arrive at
$$\frac{\p}{\p t}\(\frac{\p w}{\p t}+2w-\psi\)\geqslant \Delta\(\frac{\p w}{\p t}+2w-\psi\)-C+\frac{\p w}{\p t}+Ce^{-\frac{1}{n}\frac{\p w}{\p t}}.$$
At the spacetime minimal value point of $\frac{\p w}{\p t}+2w-\psi$, $\frac{\p w}{\p t}$ has a uniform lower bound by the maximum principle. Combining with the bounds of $w$ and $\psi$, we obtain a uniform lower bound of $\frac{\p w}{\p t}$. 

Proposition \ref{KRF-control} is thus proven. 
\end{proof}

\bibliography{mkrf-survey}

\end{document}